\documentclass{amsart}

\usepackage[T1]{fontenc}
\usepackage[utf8]{inputenc}
\usepackage[USenglish]{babel}
\usepackage{textcase}

\usepackage[
	bookmarks=true,
	plainpages=false,
	linktocpage,
	colorlinks=true,
	citecolor=green!80!black,
	linkcolor=red!70!black,
	filecolor=magenta,
	urlcolor=magenta,
	breaklinks,
	pdfauthor={Martin Winter},
]{hyperref}

\usepackage{amsmath,amsthm}
\usepackage{calc, mathtools} 

\iftrue
\usepackage{amssymb} 
\else
\usepackage[charter,cal=cmcal]{mathdesign}
\fi

\usepackage{colortbl,color} 
\usepackage[dvipsnames]{xcolor}
\usepackage{centernot} 
\usepackage{array} 
\usepackage{enumitem,moreenum} 
\usepackage{cite} 
\usepackage[nameinlink,capitalize,noabbrev]{cleveref}
\usepackage{nicefrac}
\usepackage[new]{old-arrows}

\usepackage{tikz-cd} 

\usepackage[font=small,labelfont=bf]{caption}

\usepackage{blkarray}

\usepackage{inconsolata}

\newcommand{\RR}{\mathbb{R}}    
                   
\newcommand{\F}{\mathcal{F}}    

\newcommand{\cc}{\mathfrak{c}}   
\newcommand{\oo}{\mathfrak{o}}   
\newcommand{\mm}{\mathfrak{m}}   
\newcommand{\II}{\mathfrak{I}}   

\newcommand{\hati}{\hat\imath}
\newcommand{\hatj}{\hat\jmath}

\newcommand{\Tsymb}{\top}
\newcommand{\T}{^{\Tsymb}}
\newcommand{\mT}{^{-\Tsymb}}

\def\^#1{^{(#1)}}
\def\s^#1{^{\smash{(#1)}}}

\def\:{\colon}

\newcommand{\cupdot}{\mathbin{\mathaccent\cdot\cup}}

\newcommand{\labelstyle}[1]{\upshape(\textit{#1})}
\newcommand{\mylabel}{\labelstyle{\roman*}}
\newenvironment{myenumerate}{\begin{enumerate}[label=\mylabel]}{\end{enumerate}}

\def\itm#1{\labelstyle{\romannumeral#1\relax}}

\makeatletter
\newcommand{\myitem}[1]{%
\item[#1]\protected@edef\@currentlabel{#1}%
}
\makeatother

\newcommand{\freespace}{\kern.07em}

\newcommand{\enquote}[1]{``#1''}                                 

\newcommand{\ul}[1]{\underline{\smash{#1}}}

\newcommand{\msays}[1]{{\footnotesize\textcolor{red}{#1}}}
\newcommand{\TODO}{\msays{TODO}}

\theoremstyle{plain}  
\newtheorem{theorem}{Theorem}[section]
\newtheorem{corollary}[theorem]{Corollary}
\newtheorem{lemma}[theorem]{Lemma}
\newtheorem{proposition}[theorem]{Proposition}
\newtheorem{conjecture}[theorem]{Conjecture}

\theoremstyle{definition} 
\newtheorem{definition}[theorem]{Definition}
\newtheorem{example}[theorem]{Ex\-am\-ple}
\newtheorem{remark}[theorem]{Remark}
\newtheorem{question}[theorem]{Question}
\newtheorem{observation}[theorem]{Observation}
\newtheorem{construction}[theorem]{Construction}

\crefname{theorem}{Theorem}{Theorems}
\crefname{proposition}{Proposition}{Propositions}
\crefname{lemma}{Lemma}{Lemmas}
\crefname{corollary}{Corollary}{Corollaries}
\crefname{remark}{Remark}{Remarks}
\crefname{example}{Example}{Examples}
\crefname{definition}{Definition}{Definitions}
\crefname{problem}{Problem}{Problems}
\crefname{observation}{Observation}{Observation}
\crefname{construction}{Construction}{Construction}

\DeclareMathOperator{\conv}{conv}

\DeclareMathOperator{\Aut}{Aut}

\DeclareMathOperator{\OO}{O}

\DeclareMathOperator{\GL}{GL}
\DeclareMathOperator{\PGL}{PGL}

\DeclareMathOperator{\Span}{span}

\DeclareMathOperator{\rank}{rank}

\DeclareMathOperator{\Id}{Id}

\DeclareMathOperator{\Sym}{Sym}
\DeclareMathOperator{\Perm}{Perm}   	

\DeclareMathOperator{\vol}{vol}  	
\DeclareMathOperator{\Int}{int}  	
  	
\DeclareMathOperator{\Orb}{Orb}

\let\<=\langle
\let\>=\rangle
\let\x=\times

\def\...{...}
\newcommand{\shortStyle}{\textit}
\newcommand{\ie}{\shortStyle{i.e.,}}

\newcommand{\eg}{\shortStyle{e.g.}}

\newcommand{\wrt}{\shortStyle{w.r.t.}}

\newcommand{\cf}{\shortStyle{cf.}}

\newcommand{\resp}{resp.}

\let\angle=\measuredangle

\makeatletter
\renewcommand*{\eqref}[1]{%
  \hyperref[{#1}]{\textup{\tagform@{\ref*{#1}}}}%
}
\makeatother

\numberwithin{equation}{section}

\def\nlspace{\nolinebreak\space}

\begin{document}


\expandafter\title
{Capturing Polytopal Symmetries by Coloring the Edge-Graph}
		
\author[M. Winter]{Martin Winter}
\address{Faculty of Mathematics, University of Technology, 09107 Chemnitz, Germany}
\email{martin.winter@mathematik.tu-chemnitz.de}
	
\subjclass[2010]{51M20, 52B05, 52B11, 52B15, 05C50}
\keywords{convex polytopes, linear symmetries, orthogonal symmetries, edge-graph, graph coloring, graph symmetries}
		
\date{\today}
\begin{abstract}
A general (convex) polytope $P\subset\RR^d$ and its edge-graph $G_P$ can have very distinct symmetry properties.
We construct a coloring (of the vertices and edges) of the edge-graph so that the combinatorial symmetry group of the colored edge-graph is isomorphic (in a natural way) to $\Aut_{\GL}(P)$, the group of linear symmetries of  the polytope.
We also construct an analogous coloring for $\Aut_{\OO}(P)$, the group of orthogonal symmetries of $P$.
\end{abstract}

\maketitle

\section{Introduction}
\label{sec:introduction}


In the context of this article, a \emph{polytope} $P\subset\RR^d$ will always be a \emph{convex} polytope, that is, $P$ is the convex hull of finitely many points.
A \emph{symmetry} of $P$~is a certain transformation of the ambient space that fixes the polytope set-wise.
Our focus is specifically on the groups
\begin{align*}
    \Aut_{\GL}(P) &:= \{T\in\GL(\RR^d)\mid TP=P\},\;\text{and}\\
    \Aut_{\OO}(P) &:= \{T\in\OO(\RR^d)\mid TP=P\},
\end{align*}
called the \emph{linear} \resp\ \emph{orthogonal symmetry group} of $P$. 

Initially defined geometrically, one can ask whether it is possible to understand these symmetry groups combinatorially.
This could mean to identify a purely combinatorial object $\mathcal C$ whose combinatorial symmetry group $\Aut(\mathcal C)$ is isomorphic to $\Aut_{\GL}(P)$ \resp\ $\Aut_{\OO}(P)$ in a natural way.

For example, consider the edge-graph $G_P$ of the polytope.
Every, say, linear~symmetry $T\in\Aut_{\GL}(P)$ induces a distinct combinatorial symmetry $\sigma_T\in\Aut(G_P)$~of the edge-graph (see \cref{fig:hex_symmetry}).
We could state this as follows: the edge-graph is at least as symmetric as the polytope.
Usually however, it is strictly more symmetric and is~therefore~un\-suited for \enquote{capturing the polytope's symmetries} in our sense.
\begin{figure}[h!]
    \centering
    \includegraphics[width=0.43\textwidth]{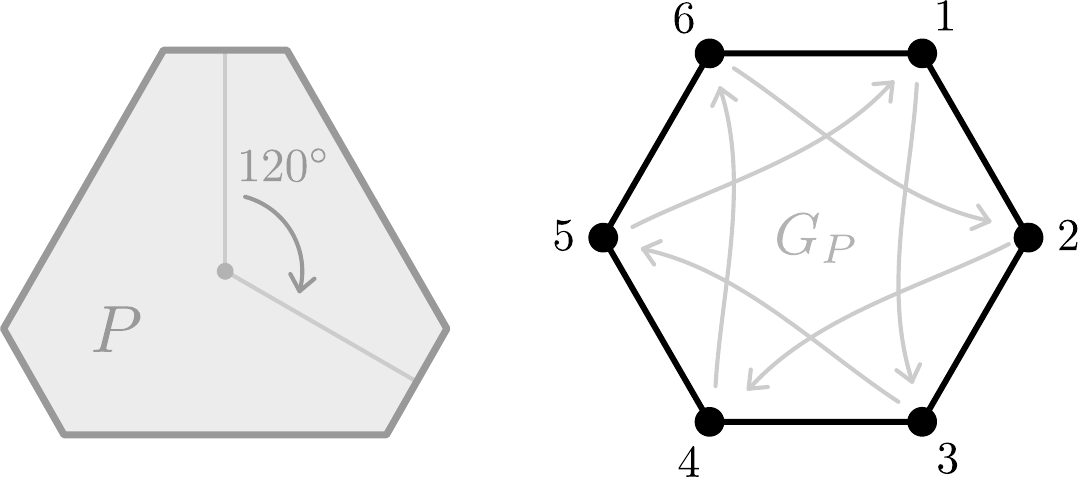}
    \caption{
    The clockwise $120^\circ$-ro\-tational symmetry of the hexagon permutes its vertices. 
    This permutation correspond to a \mbox{combinatorial~sym}\-metry $\sigma=(135)(246)$ of the edge-graph. 
    Not every combinatorial symmetry of $G_P$ comes from such a geometric symmetry, \eg\ $(123456)\in\Aut(G_P)$. The polygon is therefore strictly less symmetric than its edge-graph.
    }
    \label{fig:hex_symmetry}
\end{figure}

In this article we ask whether this can be fixed by coloring the vertices and edges of the edge-graph, thereby encoding further geometric information, and hopefully creating a combinatorial objects that is exactly as symmetric as $P$ (see \cref{fig:hex_symmetry_2}).\nolinebreak\space As we shall see, this is indeed possible.
\begin{figure}[h!]
    \centering
    \includegraphics[width=0.65\textwidth]{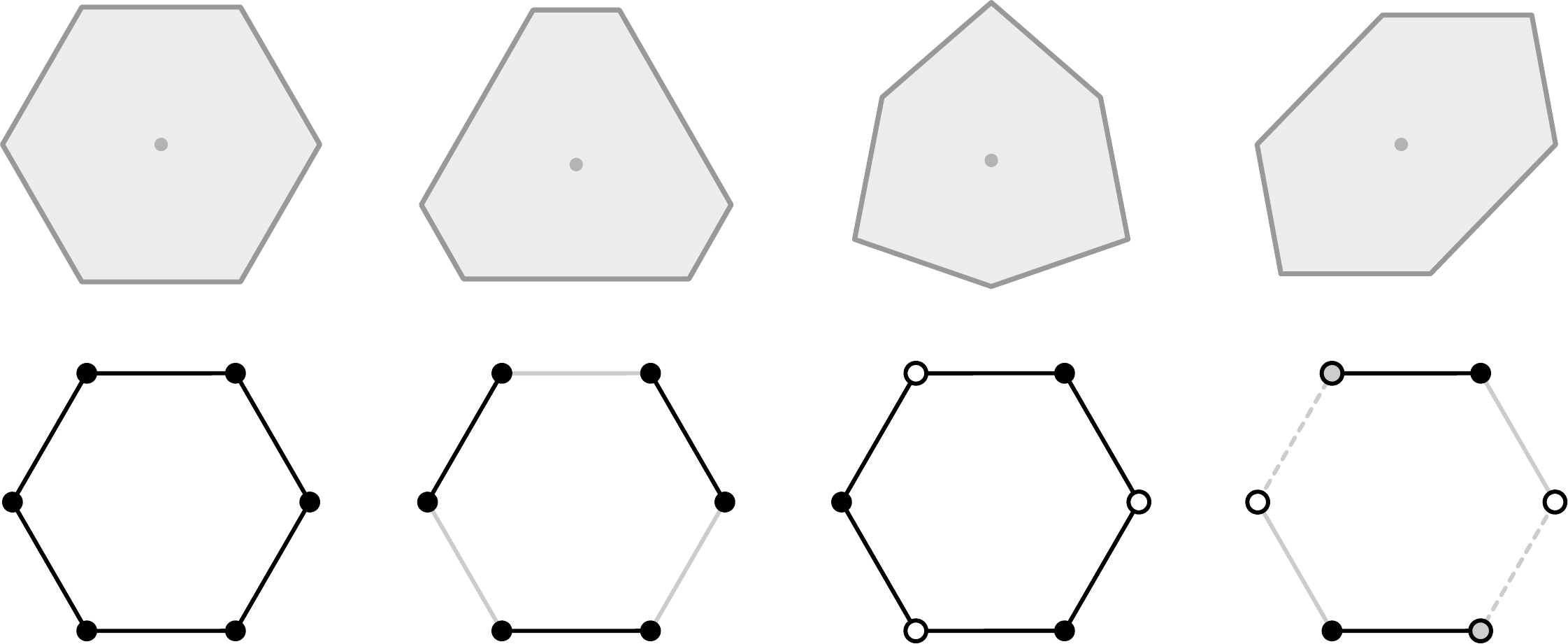}
    \caption{Various hexagons, and to each a coloring of its edge-graph that gives it \enquote{the same symmetries} as the polygon.}
    \label{fig:hex_symmetry_2}
\end{figure}

This should be surprising for at least two reasons.
First, it is established wisdom that the edge-graph of a general polytope in dimension $d\ge 4$ carries only~very little information about the polytope (a graph can be the edge-graph of several combinatorially distinct polytopes,\nolinebreak\space\mbox{potentially~of} different dimensions). 
Thus, whether the geomet\-ric symmetries of $P$ can~be~captured by coloring only the edges and vertices of $P$ (instead of, say, also higher dimensional faces) should be at~least~controversial.
Second, the same statement is actually wrong for more general geometric objects (such as graph embeddings, see \cref{ex:K_4_4}).
In fact, our proof for the~existence~of these colorings is based on a construction by Ivan Izmestiev \cite{izmestiev2010colin}, which relies heavily on the \emph{convexity} of $P$. Because of this, it is unclear whether our result generalizes to even some form of non-convex polytopes or polytopal complexes.

Our investigation is in part motivated from a result by Bremner et al.~\cite{bremner2014computing}: given a polytope $P\subset\RR^d$ with $n$ vertices, the authors construct a coloring of the complete graph $K_n$, so that the symmetry group of the colored graph is isomorphic to $\Aut_{\GL}(P)$ (\resp\ $\Aut_{\OO}(P)$; a more precise statement is given in \cref{ssec:metric_coloring}).~%
We can interpret this as follows:
if we are allowed to color not only the vertices and edges of $P$, but also other pairs of vertices without a direct counterpart in the po\-ly\-tope's combinatorics, then \enquote{capturing the polytope's symmetries} is indeed possible.
The major result of our article is then that coloring these \enquote{non-geometric edges} is not actually necessary.

We reiterate this introduction in a more formal manner.

\subsection{Notation and setting}
\label{ssec:setting}


Throughout the text we let $P\subset\RR^d$ denote a convex polytope that~is full-dimensional (\ie\ not con\-tained in any proper affine subspace of $\RR^d$) and~contains the origin in its interior (\ie\ $0\in\Int(P)$).

By $\F_\delta(P)$ we denote the set of $\delta$-dimensional faces of $P$. We assume a fixed enumeration $v_1,...,v_n\in\F_0(P)$ of the polytope's vertices.
In particular, $n$ will~always denote the number of the vertices.

The edge-graph of $P$ is the finite simple graph $G_P=(V,E)$ with~vertex set~$V=\{1,...,n\}$ and edge set $\smash{E\subseteq{V\choose 2}}$.
We implicitly assume that $i\in V$ corresponds~to~the vertex $v_i\in\F_0(P)$, and that $ij\in E$ (short for $\{i,j\}\in E$) if and only if $\conv\{v_i,v_j\}\in\F_1(P)$.


The (combinatorial) symmetry group of $G_P$%
\footnote{For convenience, notions like the symmetry group, colorings, the adjacency matrix, etc.\  are only introduced for the edge-graph, but it is understood that they apply to more general graphs as well.} 
is defined as
$$\Aut(G_P):=\{\sigma\in\Sym(V)\mid ij\in E\Leftrightarrow \sigma(i)\sigma(j)\in E\}\subseteq\Sym(V)\text{\footnotemark},$$
\footnotetext{$\Sym(V)$ denotes the \emph{symmetric group}, \ie\ the group of permutations of the set $V$.}
that is, the group of permutations of $V$ that fix the edge set of $G_P$.

A \emph{coloring} of $G_P$ is a map $\cc\:V\cupdot E\to\mathfrak C$ (it assign colors to both, vertices and edges), where $\mathfrak C$ denotes an abstract \emph{set of colors}.
The pair $(G_P,\cc)$ is then a \emph{colored edge-graph} and will be abbreviated by $G_P^\cc$.
Its combinatorial symmetry group is
$$\Aut(G_P^\cc):=\Big\{\sigma\in\Aut(G_P) \;\Big\vert {\small\begin{array}{rcll}
\cc(i)  &\!\!\!\!\!=\!\!\!\!\!&  \cc(\sigma(i)) &\text{for all $i\in V$ } \\
\cc(ij) &\!\!\!\!\!=\!\!\!\!\!&  \cc(\sigma(i)\sigma(j)) &\text{for all $ij\in E$}
\end{array}}\Big\}.$$
If $\sigma\in\Aut(G_P^\cc)$, we also say that $\sigma$ \emph{preserves} the coloring $\cc$.

The \emph{colored adjacency matrix} of $G_P^\cc$ is the matrix $A^\cc\in (\mathfrak C\cupdot\{0\})^{n\x n}$ with entries
$$A^\cc_{ij}:=\begin{cases} 
\cc(i) & \text{if $i=j$}\\
\cc(ij) & \text{if $ij\in E$}\\
0 & \text{otherwise}
\end{cases}.
$$
Clearly, a coloring is completely determined by the colored adjacency matrix, and we might occasionally use $A^\cc$ to define a coloring.

%



A geometric symmetry $T\in\Aut_{\GL}(P)$ of $P$ maps vertices of $P$ onto vertices~of~$P$ and thus describes a permutation of the vertex set.
Let $\sigma_T\in\Sym(V)$ be the~permutation of the vertex set of the edge-graph that permutes its vertices in the same way as $T$ permutes the vertices of $P$.
Formally, that is
\begin{equation}
\label{eq:sigma_T}
 T v_i = v_{\sigma_T(i)},\quad\text{for all $i\in V$}.
\end{equation}
Since $T$ also maps edges of $P$ onto edges of $P$, also $\sigma_T$ maps edges to edges, and so we see that $\sigma_T$ is a symmetry of the edge-graph, \ie\ $\sigma_T\in\Aut(G_P$).
%
The assignment $T\mapsto \sigma_T$ then defines a group homomorphism $\phi\:\Aut_{\GL}(P)\to\Aut(G_P)$ which we shall call the \emph{natural group homomorphism} of the polytope $P$. 

Since $P$ is full-dimensional, its vertices contain a basis of $\RR^d$, and it follows that $\phi$ must be \emph{injective}. 
In general however, $\phi$ is not an isomorphism and $\Aut_{\GL}(P)\not\cong \Aut(G_P)$, which is a formal way to say that the edge-graph $G_P$ can have many more symmetries than the polytope.

Our approach for rectifying this is to assign a coloring $\cc\:V\cupdot E\to\mathfrak C$ to the edge-graph $G_P$ with the hope that $\Aut_{\GL}(P)\cong\Aut(G_P^\cc)$. The natural candidate for the isomorphism between the groups is a colored version of the natural homomorphism:
\begin{equation}
\label{eq:natural_homo}
    \phi^\cc\:\Aut_{\GL}(P)\to\Aut(G_P^\cc),\;\,T\mapsto \sigma_T
\end{equation}
%
%

For this to work as desired, we need to check two things:
\begin{itemize}
    \item First, $\phi^\cc$ needs to be well-defined. This is not the case for each coloring: one needs to check that for each $T$ $\in$ $\Aut_{\GL}(P)$ the corresponding permutation $\sigma_T$ is indeed a symmetry of $G_P^\cc$ (that is, is in $\Aut(G_P^\cc)$). 
    Intuitively, this amounts to checking that the edge-graph, even after coloring, is still at \emph{least} as symmetric as $P$.
    \item Second, $\phi^\cc$ must have an inverse. If so, then $G_P^\cc$ is \emph{exactly} as symmetric as $P$. Providing such an inverse will go as follows: for each $\sigma\in\Aut(G_P^\cc)$ we need to construct a geometric symmetry $T_\sigma\in\Aut_{\GL}(P)$ with
    $$T_\sigma v_i = v_{\sigma(i)},\quad\text{for all $i\in V$}.$$
    Since $P$ is full-dimensional, if $T_\sigma$ exists then it is unique.
    The map $\sigma\mapsto T_\sigma$ is then the desired inverse.
\end{itemize}



The discussion also applies verbatim to the orthogonal symmetry group $\Aut_{\OO}(P)$, and we shall use the same notation $\phi^\cc\:\Aut_{\OO}(P)\to\Aut(G_P^\cc)$ to denote the natural homomorphism in this case. 

With this in place, we can formalize \enquote{capturing symmetries}:

\begin{definition}
\label{def:capturing_symmetries}
A coloring $\cc\:V\cupdot E\to\mathfrak C$ of $G_P$ is said to \emph{capture the linear (\resp\ orthogo\-nal) symmetries} of $P$ if $\Aut(G_P^\cc)\cong \Aut_{\GL}(P)$ (\resp\ $\Aut(G_P^\cc)\cong\Aut_{\OO}(P)$), where the isomorphism is realized by the natural homomorphism $\phi^\cc$.
\end{definition}


The main results of this article are explicit constructions for colorings that
\begin{itemize}
    \item capture linear symmetries (\cref{res:Izmestiev_works}).
    \item capture orthogonal symmetries (\cref{res:capturing_Euclidean_symmetries}).
\end{itemize}

\subsection{Overview}


In \cref{sec:colorings} we introduce the \emph{metric coloring} and the \emph{orbit color\-ing}, two very natural candidates for capturing certain polytopal symmetries.
In this section we do not yet show that either coloring capture linear or orthogonal symmetries, but we establish relevant properties used in the upcoming sections. 


In \cref{sec:linear_algebra_condition} we derive a sufficient condition for a coloring of the form $\cc\: V\cupdot E\to\RR$ (the colors are real numbers) to capture linear symmetries.
The criterion will~be in terms of the eigenspaces of the (colored) adjacency matrix of the edge-graph.\nolinebreak\space We shall call this the \enquote{linear algebra criterion}.

In \cref{sec:izemestiev} we introduce the \emph{Izmestiev coloring} (based on a construction by Ivan Izmestiev \cite{izmestiev2010colin}) and we show that it satisfies the \enquote{linear algebra criterion} from \cref{sec:linear_algebra_condition}.
We~thereby establish the existence of a first coloring that captures linear symmetries (\cref{res:Izmestiev_works}).
As a corollary we find that the orbit coloring captures linear symmetries as well (\cref{res:orbit_coloring_works}).

In \cref{sec:capturing_Euclidean_symmetries} we show that a combination of the Izmestiev coloring and the metric coloring captures orthogonal symmetries (\cref{res:capturing_Euclidean_symmetries}).

\section{Two useful colorings}
\label{sec:colorings}

This section is preliminary, in that it introduce two natural colorings of the~edge-graph, the \emph{metric coloring} and the \emph{orbit coloring}, without establishing either color\-ing as capturing polytopal symmetries. 
In fact, this is an open question for the~metric coloring (see \cref{q:metric_coloring}).
The orbit coloring captures polytopal symmetries, but we are not able to show this right away.
Both colorings will play a role in~the~up\-coming sections.


\Cref{fig:both_colorings} shows a polygon and its edge-graph with either coloring applied.



\begin{figure}[h!]
    \centering
    \includegraphics[width=0.5\textwidth]{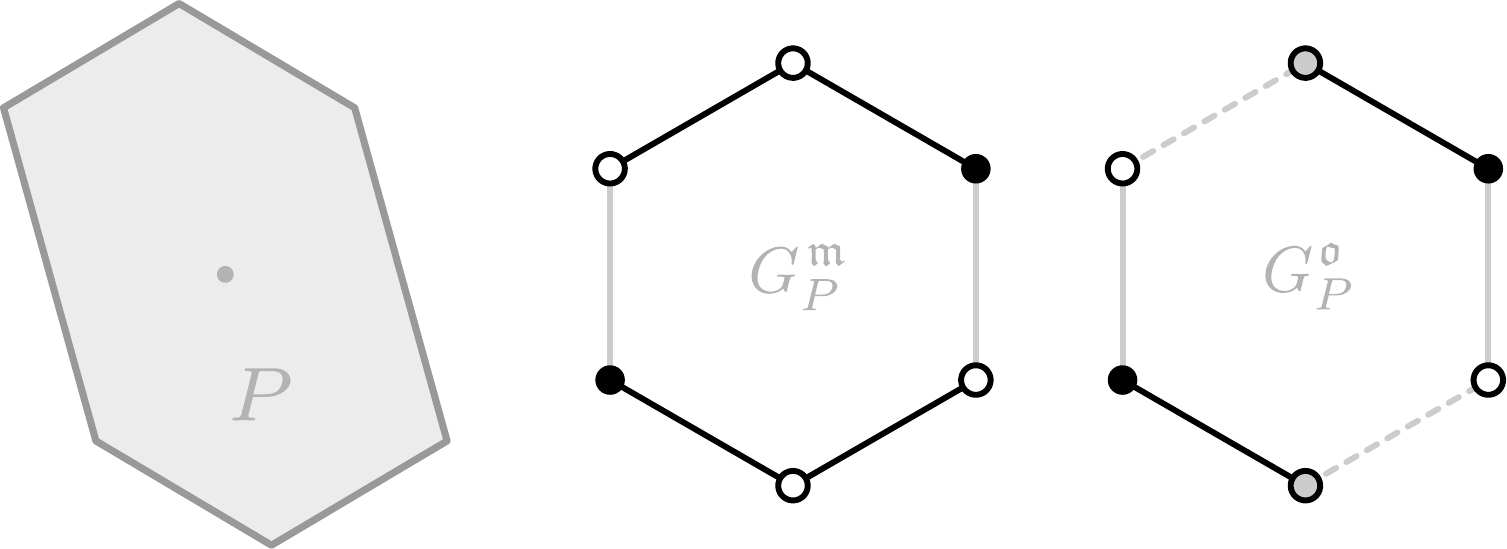}
    \caption{\mbox{A hexagon and its edge-graph colored with the metric~coloring} (middle, \cref{ssec:metric_coloring}) \resp\ the orbit coloring (right, \cref{ssec:orbit_coloring}).}
    \label{fig:both_colorings}
\end{figure}

\subsection{The metric coloring}
\label{ssec:metric_coloring}

Our first coloring is motivated from the previously~mentioned construction of Bremner et al. \cite{bremner2014computing} -- a coloring of the complete graph $K_n$ that \enquote{captures orthogonal symmetries}.
In our notation their result reads as follows:


\begin{theorem}[\!$\text{\cite[Theorem 2]{bremner2014computing}}$\footnote{This result is primarily based on \cite[Proposition 3.1]{bremner2009polyhedral}, but we found that its first explicit~formulation is in \cite{bremner2014computing}}\,]
\label{res:complete_coloring}
Given a polytope $P\subset\RR^d$ with vertex set $\F_0(P) = \{v_1,...,v_n\}$. 
\mbox{Consider} the coloring $\cc$ on the complete graph $K_n$ with 
\begin{align*}
\qquad\qquad\qquad\qquad \mathfrak c(i)&:=\|v_i\|^2,\quad &&\text{for all $i\in \{1,...,n\}$},\qquad\qquad\qquad\qquad\qquad\qquad\\
\qquad\qquad\qquad\qquad \mathfrak c(ij)&:=\<v_i,v_j\>,\quad &&\text{for all distinct $i,j\in \{1,...,n\}$}.\qquad\qquad\qquad\qquad\qquad\qquad
\end{align*}
Then $\Aut(K_n^\cc)\cong \Aut_{\OO}(P)$.
\end{theorem}

The strength of this result lies in its immediate applicability: constructing this \enquote{complete metric coloring} requires no knowledge of the edge-graph (which is usually hard to come by), but only the vertex coordinates of $P$\,\footnote{If $P$ is given in $\mathcal H$-representation, one can apply \cref{res:complete_coloring} to compute the orthogonal~symmetry group of the dual polytope $P^\circ$, which is identical to $\Aut_{\OO}(P)$ as a matrix group.}.
In practice, this~is~probably the best tool for an explicit computation of $\Aut_{\OO}(P)$.
%
%


From a theoretical and aesthetic perspective however, this construction has the flaw of containing massively redundant data and stepping outside the combinatorial structure of the polytope (we assign color to vertex-pairs that are not~edges~of~the~po\-lytope).
Naturally, we can ask whether one can get away with coloring fewer of~these \enquote{non-edges}, ideally only the actual edges of the edge-graph.


Based on this hope, we define the following:

\begin{definition}
\label{def:metric_coloring}
The \emph{metric coloring} of $G_P$ is the coloring $\mathfrak m\:V\cupdot E\to \RR$\,\footnote{
A coloring whose colors are real numbers is still a \emph{purely combinatorial objects}.
These numbers are just used for a concise definition and could be replaced by any other finite set of distinguishable values. 
The only information used from the coloring (in the form of the combinatorial symmetry group of the colored graph) is whether two vertices/edges receive the same or a different color.
} with
%
\begin{align*}
\qquad\qquad\qquad\qquad \mathfrak m(i)&:=\|v_i\|^2,\quad &&\text{for all $i\in V$},\qquad\qquad\qquad\qquad\qquad\qquad\\
\qquad\qquad\qquad\qquad \mathfrak m(ij)&:=\<v_i,v_j\>,\quad &&\text{for all $ij\in E$}.\qquad\qquad\qquad\qquad\qquad\qquad
\end{align*}
%
\end{definition}

Whether the metric coloring captures orthogonal symmetries is an open question (see also \cref{q:metric_coloring}).
Our reason for introducing it anyway is that in \cref{sec:capturing_Euclidean_symmetries} the metric coloring will be one ingredient to a coloring that indeed captures orthogonal symmetries.

We close this section with another formulation of \cref{res:complete_coloring} that also allows for capturing linear symmetries  (in fact, this is closer to the original formulation in \cite{bremner2014computing}).
Note that the complete metric coloring of $K_n$ in \cref{res:complete_coloring} can also be described by its colored adjacency matrix $A^\cc = \Phi\T\Phi$, where $\Phi:=(v_1,...,v_n)\in\RR^{d\x n}$ is~the matrix in which the vertex coordinates of $P$ appear as columns.

\begin{theorem}[Another formulation of \!$\text{\cite[Theorem 2]{bremner2014computing}}$]
\label{res:complete_coloring_linear}
Let $\cc$ be a coloring of the complete graph $K_n$ with colored adjacency matrix $A^\cc$:
\begin{myenumerate}
    \item if $A^\cc=\Phi\T\Phi$, then $\Aut(K_n^\cc)\cong\Aut_{\OO}(P)$ (this is exactly \cref{res:complete_coloring}).
    \item if $A^\cc=\Phi^\dagger\Phi$\,\footnote{$\Phi^\dagger\in\RR^{n\x d}$ denotes the \emph{Moore-Penrose pseudo inverse} of $\Phi$, that is, $\Phi\Phi^\dagger=\Id_d$.}, then $\Aut(K_n^\cc)\cong\Aut_{\GL}(P)$.
\end{myenumerate}
\end{theorem}

A proof for part \itm2 will also follow from the theory developed in \cref{sec:linear_algebra_condition} (see \cref{ex:easy_wrong_example})

\subsection{The orbit coloring}
\label{ssec:orbit_coloring}

The next coloring is motivated from the following~con\-sideration:
%
suppose that we are given two vertices $v_i,v_j\in\F_0(P)$ in the same~orbit \wrt\ $\Aut_{\GL}(P)$, which just means that there is a $T\in\Aut_{\GL}(P)$ with \mbox{$Tv_i = v_j$}.~%
The corresponding combinatorial symmetry $\sigma_T\in\Aut(G_P)$ satisfies $\sigma_T(i)=j$.~%
If~now $\cc\:V\cupdot E\to\mathfrak C$ is a coloring that captures linear symmetries%
, then $\sigma_T$ preserves the coloring $\cc$ and we have $\cc(j)=\cc(\sigma_T(i))=\cc(i)$.
\mbox{We can summarize} this as follows:~if~$\cc$ is supposed to capture linear symmetries, then vertices in the same $\Aut_{\GL}(P)$-or\-bit of $P$ must have the same color in $G_P^\cc$. %
With an analogous~argument we see that the same holds for edges.




Having identified this first necessary condition for capturing symmetries, we can consider the \enquote{simplest} coloring that follows this idea:

\begin{definition}
\label{def:orbit_coloring}
The (\emph{linear}) \emph{orbit coloring} $\mathfrak o$ of $G_P$ assigns the same color to ver\-ti\-ces (\resp\ edges) of $G_P$ if and only if the corresponding vertices (\resp\ edges) of~$P$ are in the same $\Aut_{\GL}(P)$-orbit. 
%
\end{definition}

An analogous coloring can be defined for orthogonal symmetries, which we shall call the \emph{orthogonal orbit coloring} of $G_P$, still denoted by $\oo$. 
For the sake of conciseness, this section only discusses the (linear) orbit coloring, but all statements carry over to the orthogonal version in the obvious way.


As we shall learn in \cref{sec:izemestiev} (see \cref{res:orbit_coloring_works}), the orbit coloring indeed~captures linear symmetries.
However, this is surprisingly hard to show directly.
In~fact, our eventual proof of this will \enquote{just} use the following:

\iftrue

\begin{lemma}
\label{res:orbit_coloring_finest}
If there is any coloring that captures linear symmetries, then so does the orbit coloring $\oo$.
%
%
\begin{proof}
Suppose that $\cc$ is a coloring that captures linear symmetries, in particular,~$\phi^\cc$ is an isomorphism.
Our proof that $\oo$ captures linear symmetries as well is based on two simple observations:
\begin{myenumerate}
    \item the natural homomorphism $\phi^\oo$ is well-defined (that is, $G_P^\oo$ is at least as~symmetric as $P$), and
    \item $\Aut(G_P^\oo)\subseteq\Aut(G_P^\cc)$.
\end{myenumerate}
Showing either is straight-forwarded, but for the sake of completeness, both proofs are included below. 
Now, presupposing both, we can write down~the following~chain of groups in which the first and the last group are the same:
$$
\begin{tikzcd}
\mathrm{Aut}(G_P^{\mathfrak c}) \arrow[r, "(\phi^\cc)^{-1}"] & \mathrm{Aut}_{\mathrm{GL}}(P) \arrow[r, "\phi^{\mathfrak o}"] & \mathrm{Aut}(G_P^{\mathfrak o}) \arrow[r, hook, "\text{\itm2}"] & \mathrm{Aut}(G_P^{\mathfrak c})
\end{tikzcd}\,.
$$
Since all maps are injective, and the groups are finite, all maps must actually~be~iso\-morphisms.
Thus, $\phi^{\mathfrak o}$~is~an~iso\-morphism and $\oo$ captures linear symmetries.
This~concludes the proof, and it remains to verify \itm1 and \itm2.

Proof of \itm1: let $T\in\Aut_{\GL}(P)$ be a linear symmetry of $P$ with corresponding combinatorial symmetry $\sigma_T\in\Aut(G_P)$. 
We need to show that $\sigma_T\in\Aut(G_P^\oo)$.
For this, we observe that for each $i\in V$ the vertices $v_i$ and $v_{\sigma_T(i)}=T v_i$ belong~to~the same $\Aut_{\GL}(P)$-orbit of $P$.
By the definition of the orbit coloring, $i$ and $\sigma_T(i)$ have then the same color in $G_P^\oo$.
Thus, $\sigma_T$ preserves the vertex colors~of $\oo$.~%
Analogous\-ly, one shows that $\sigma_T$ preserves edge colors.
Thus, $\sigma_T\in\Aut(G_P^{\mathfrak o})$.

Proof of \itm2: let $\sigma\in\Aut(G_P^\oo)$ be a permutation that preserves the orbit coloring. We need to show $\sigma\in\Aut(G_P^\cc)$.
For this, we observe that for all $i\in V$ the~vertices $i$ and $\sigma(i)$ have the same color in $G_P^\oo$, which just means (by \cref{def:orbit_coloring}) that $v_i,v_{\sigma(i)}\in\F_0(P)$ are in the~same $\Aut_{\GL}(P)$-orbit of $P$. 
Repeating the argument of the introductory paragraph to~this section we see that $\cc(i)=\cc(\sigma(i))$.
An analogous argument holds for edges.
In~other words,~$\sigma$ preserves the coloring $\cc$, and hence $\sigma\in\Aut(G_P^\cc)$.
\end{proof}
\end{lemma}



\else

\begin{proposition}
\label{res:orbit_coloring_finest}
If any coloring captures linear symmetries, then so does the orbit coloring.
In that case, the orbit coloring is the finest coloring that \mbox{captures~linear~sym}\-metries.
\begin{proof}
By \cref{res:coarsening_orbit_coloring_makes_well_defined_homomorphism}, the orbit coloring $\mathfrak o$ comes with a well-defined natural homomorphism $\phi^{\mathfrak o}\:\Aut_{\GL}(P)\to\Aut(G_P^{\mathfrak o})$.



If $\cc$ is any coloring that captures linear symmetries, then we obtain the following cyclic chain of groups:
$$
\begin{tikzcd}
\mathrm{Aut}(G_P^{\mathfrak c}) \arrow[r, "\cong"] & \mathrm{Aut}_{\mathrm{GL}}(P) \arrow[r, "\phi^{\mathfrak o}"] & \mathrm{Aut}(G_P^{\mathfrak o}) \arrow[r, hook, "\eqref{eq:containment}"] & \mathrm{Aut}(G_P^{\mathfrak c})
\end{tikzcd}
$$
Since all maps are injective, they must actually be isomorphisms, including $\phi^{\mathfrak o}$.

By \cref{res:must_coarsen}, the orbit coloring must then be the finest coloring that~captures linear symmetries.
\end{proof}
\end{proposition}

\fi

\section{A linear algebra condition for capturing symmetries}
\label{sec:linear_algebra_condition}

For this section, fix a coloring $\cc\:V\cupdot E\to\mathfrak C$ for which $G_P^\cc$ is at least as symmetric as $P$. Then $\phi^\cc\:\Aut_{\GL}(P)\to\Aut(G_P^\cc)$ is well-defined.
The goal of this section~is~to derive a sufficient criterion for $\cc$ to capture linear symmetries.

Recall that this amounts to showing that $\phi^\cc$ is an isomorphism.
In other words, the desired criterion must ensure that for each $\sigma\in\Aut(G_P^\cc)$ we can find a linear symmetry $T_\sigma\in\Aut_{\GL}(P)$ with
\begin{equation}
\label{eq:isomorphism_separate}
T_\sigma v_i = v_{\sigma(i)},\quad\text{for all $i\in V$}.
\end{equation}
Let us investigate the difficulties in constructing these transformations.

%
%
%
%
%
%
First, note that we can express \eqref{eq:isomorphism_separate} for all $i\in V$ simultaneously by rewriting it into a single matrix equation as follows:
%
%
$$T_\sigma(v_1,...,v_n)=(v_{\sigma(1)},...,v_{\sigma(n)}) = (v_1,...,v_n)\Pi_\sigma,$$
where $\Pi_\sigma\in\Perm(n)$ denotes the corresponding permutation matrix\footnote{We chose to define $\Pi_\sigma$ so that on multiplication from \emph{left} it permutes the \emph{rows} as prescribed by $\sigma$. We emphasize that this, counter-intuitively, means $(\Pi_\sigma v)_i=v_{\sigma^{-1}(i)}$ for a vector $v\in\RR^n$.}. 
If we define $\Phi$ $:=$ $(v_1,...,v_n)\in\RR^{d\x n}$ as the matrix in which the polytope's vertices $v_i$ appear~as columns, this further compactifies to
\begin{equation}
\label{eq:equivariant}
T_\sigma \Phi=\Phi\Pi_\sigma.
\end{equation}
This equation will be our benchmark: every ansatz for how to define the transforma\-tions $T_\sigma$ must satisfy \eqref{eq:equivariant}, which is then also sufficient.

%
%
Now, if $\Phi$ were invertible, we could just solve \eqref{eq:equivariant} for $T_\sigma$, satisfying \eqref{eq:equivariant} \enquote{by force}.
However, $\Phi\in\RR^{d\x n}$ is not a square matrix (since $P$ \mbox{is full-dimen}\-sio\-nal,\nolinebreak\space we have $n\ge d+1$).\nolinebreak\space
%
%
Instead, one naive hope to still \enquote{solve for $T_\sigma$} is to use the \emph{Moore-Penrose pseudo inverse} of $\Phi$:~the unique matrix~$\Phi^\dagger$ $\in$ $\RR^{n\x d}$ with  $\Phi\Phi^\dagger=\Id_d$ (the rows of $\Phi^\dagger$ form a dual basis to the columns of $\Phi$).
And so we make the~following~an\-satz:
\begin{equation}
\label{eq:def_T}
T_\sigma:=\Phi\Pi_\sigma\Phi^\dagger.
\end{equation}
%

It remains to investigate under which conditions this ansatz satisfies \eqref{eq:equivariant}. We compute
\begin{equation}
\label{eq:subst_T}
 T_\sigma \Phi \overset{\eqref{eq:def_T}} = \Phi\Pi_\sigma\Phi^\dagger \Phi = \Phi\Pi_\sigma \pi_U,
\end{equation}
where $\pi_U:=\Phi^\dagger\Phi$ is the orthogonal projector onto the subspace $U:=\Span\Phi^\dagger\subseteq\RR^n$.
Apparently, to arrive at \eqref{eq:equivariant}, we would need to get rid of the projector~$\pi_U$~on~the right side of \eqref{eq:subst_T}.
And so we see that one possible sufficient criterion for our~construction of the $T_\sigma$ to work (and thus, for $\cc$ to capture linear symmetries) would be $\Phi\Pi_\sigma\pi_U = \Phi\Pi_\sigma$ for all $\sigma\in\Aut(G_P^\cc)$.

This is still a rather cumbersome criterion to apply.
The main result of this~section is then to reformulate this in terms of the adjacency matrix of $G_P^\cc$.



\begin{theorem}
\label{res:criterion}
Let $\cc\:V\cupdot E\to\RR$ be a coloring of the edge-graph $G_P$ so that $G_P^\cc$ is at least as symmetric as $P$.
If $U:=\Span \Phi^\dagger$ is an eigenspace of the colored~adjacency matrix $A^\cc$,
then $\cc$ captures the linear symmetries of $P$.
\begin{proof}
Fix a combinatorial symmetry $\sigma\in\Aut(G_P^\cc)$.

We use the following well-known (and easy to verify) property of the colored~adjacency matrix: if $\sigma\in\Aut(G_P^\cc)$, then
$$\Pi_\sigma A^\cc = A^\cc \Pi_\sigma.$$
%
Now, if $A^\cc$ and $\Pi_\sigma$ commute, then the eigenspaces of $A^\cc$ (including $U$) are \emph{invariant subspaces} of $\Pi_\sigma$, \ie\ $\Pi_\sigma U= U$.
Equivalently, $\Pi_\sigma$ commutes with the projec\-tor~$\pi_U$.
This suffices to show that the map $T_\sigma:=\Phi\Pi_\sigma\Phi^\dagger$ satisfies \eqref{eq:equivariant}:
\vspace{-0.5ex}
$$T_\sigma\Phi 
=\Phi\Pi_\sigma\Phi^\dagger\Phi =
\Phi\Pi_\sigma \pi_U=\Phi\pi_U \Pi_\sigma 
= \Phi(\Phi^\dagger \Phi) \Pi_\sigma 
= \overbrace{(\Phi\Phi^\dagger)}^{\Id_d} \Phi \Pi_\sigma 
= \Phi \Pi_\sigma.$$

Therefore, the map $\sigma\mapsto T_\sigma$ defines the desired inverse of $\phi^\cc$, and $\cc$ captures the linear symmetries of $P$.
\end{proof}
\end{theorem}

It might not be immediately obvious how \cref{res:criterion} is a helpful reformulation of the problem.
To apply it we need to construct a matrix $A^\cc$ with \mbox{two very~special} properties:
first, $A^\cc$ must be a (colored) adjacency matrix of the edge-graph $G_P$,\nlspace that is, it must have non-zero entries only where $G_P$ has edges.
Second, we need to ensure that $A^\cc$ has $U$ as an eigenspace.
It is not even clear that these two conditions are compatible.

\begin{remark}
\label{ex:easy_wrong_example}
Consider
%
%
the \enquote{obvious} matrix $A^\cc$ with eigenspace $U:=\Span\Phi^\dagger$:
%
$$A^\cc := \Phi^\dagger\Phi.$$
Of course, this matrix has most likely no zero-entries and is therefore not a colored adjacency matrix of $G_P$ (except if $G_P$ is the complete graphs).
However, it~is~exactly the colored adjacency matrix of the complete metric coloring as discussed in \cref{res:complete_coloring_linear} \itm2.

As it turns out, the proof of \cref{res:criterion} makes no use of the fact that the~coloring $\cc$ is defined on the edge-graph.
In fact, we can apply it to the complete~graph $K_n^\cc$ with colored adjacency matrix $A^\cc$.
In this way, the \enquote{linear algebra criterion}~pro\-vides an alternative proof of  \cref{res:complete_coloring_linear} \itm2.

\end{remark}

\section{The Izmestiev coloring}
\label{sec:izemestiev}

In this section we introduce a coloring of $G_P$ which satisfies the \enquote{linear algebra condition} \cref{res:criterion}.
This coloring is based on a construction by Ivan Izmestiev \cite{izmestiev2010colin} and we shall call it the \emph{Izmestiev coloring}.

The coloring is built in a quite unintuitive way.
First, we need to recall that for a polytope $P$ with $0\in\Int(P)$ the \emph{polar dual} $P^\circ$ is defined as
$$P^\circ:=\{x\in\RR^d\mid \<x,v_i\>\le 1 \text{ for all $i\in V$}\}.$$
%
%
We generalize this notion: for a vector $c=(c_1,...,c_n)\in\RR^n$ let
\begin{equation}
\label{eq:generalized_dual}
P^\circ(c) := \{x\in\RR^d\mid \<x,v_i\>\le c_i\text{ for all $i\in V$}\}.
\end{equation}
Then $P^\circ(1,...,1)=P^\circ$ and $P^\circ(c)$ is obtained from $P^\circ$ by shifting facets along their normal vectors (see \cref{fig:cube_dual}).

\begin{figure}[h!]
    \centering
    \includegraphics[width=0.65\textwidth]{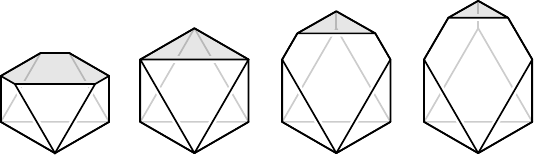}
    \caption{Several instances of the generalized dual $P^\circ(c)$ of the cube (the usual polar dual of the cube is the regular octahedron; the second from the left). The polytopes differ by a single facet-defining plane being shifted along its normal vector.}
    \label{fig:cube_dual}
\end{figure}

In the following, $\vol(C)$ denotes the \emph{relative volume} (relative to the affine hull~of $C$) of a compact convex set $C\subset\RR^d$.

\begin{theorem}[Izmestiev \cite{izmestiev2010colin}, Theorem 2.4]
\label{res:izmestiev}
For a polytope $P\subset\RR^d$ with $0\in\Int(P)$ consider the matrix $M\in$ $\RR^{n\times n}$ (which we shall call the \ul{Izmestiev matrix} of $P$) with components
$$M_{ij}:=\frac{\partial^2\vol(P^\circ(c))}{\partial c_i\partial c_j}\Big|_{\,c=(1,...,1)}.$$
(in particular, $\vol(P^\circ(c))$ is two times continuously differentiable in $c$).
$M$ then~has the following properties:
\begin{myenumerate}
	\item $M_{ij} < 0$ whenever $ij\in E$.
	\item $M_{ij} = 0$ whenever $ij\not\in E$ and $i\not=j$.
	\item $M$ has a unique negative eigenvalue of multiplicity one. 
	\item $M\Phi\T=0$, where $\Phi=(v_1,...,v_n)\in\RR^{d\x n}$ is the matrix introduced in \eqref{eq:equivariant}.
	\item $\dim \ker M=d$.
\end{myenumerate}
\end{theorem}

\begin{remark}

In the words of \cite{izmestiev2010colin}, the matrix $M$ constructed in \cref{res:izmestiev} is~a \emph{Colin~de~Ver\-dière matrix} of the edge-graph, that is, a matrix satisfying a certain list of properties, including \itm1, \itm2 and \itm3 and the so-called \emph{strong Arnold property} (for details, see \eg\ \cite{van1999colin}).

Among the Colin de Verdière matrices, one usually cares about the ones with~the largest possible kernel.
The dimension of this largest kernel is known as the \emph{Colin de Verdière graph invariant} $\mu(G_P)$ \cite{van1999colin}, and \cref{res:izmestiev} \itm5 then shows that $\mu(G_P)\ge d$.
This is not too surprising and was known before.
However, the result of Izmestiev is remarkable for a different reason: it shows that there is a Colin de Verdière matrix whose kernel has dimension \emph{exactly} $d$ (property \itm5) and that is compatible with the geometry of $P$ (property \itm4).

\end{remark}

\begin{remark}

Izmestiev also shows that the matrix $M$ can be expressed in terms of simple geometric properties of the polytope:
for $ij\in E$ let $f_{ij}\in\F_{d-2}(P^\circ)$ be the \emph{dual face} to the edge $\conv\{v_i,v_j\}\in\F_1(P)$. Then
\begin{equation}
\label{eq:izmestiev_geometic_interpretation}
M_{ij}=-\frac{\vol(f_{ij})}{\|v_i\| \|v_j\| \sin\angle(v_i,v_j)}.
\end{equation}
%



\end{remark}

\begin{definition}
The \emph{Izmestiev coloring} $\mathfrak I\:V\cupdot E\to\RR$ of $G_P$ is defined by
\begin{align*}
    \mathfrak I(i) &:= M_{ii},\quad\text{for all $i\in V$},\\
    \mathfrak I(ij) &:= M_{ij},\quad\text{for all $ij\in E$},
\end{align*}
where $M\in\RR^{n\x n}$ is the Izmestiev matrix of $P$.
\end{definition}

\begin{observation}
\label{res:Izmestiev_adjacency}
Since $M_{ij}=0$ whenever $ij\notin E$ and $i\not=j$ (by \cref{res:izmestiev}~\itm2), the colored adjacency matrix $A^\II$ of $G_P^{\mathfrak I}$ is exactly the Izmestiev matrix $M$.
\end{observation}

In order to apply the \enquote{linear algebra criterion} from  \cref{sec:linear_algebra_condition}, showing that~$\phi^{\mathfrak I}$~is an isomorphism, we first need to show that $\phi^{\mathfrak I}$ is well-defined, that is, that $G_P^{\mathfrak I}$ is at least as symmetric as $P$.
This part is relatively straightforward if we use that the Izmestiev matrix is a linear invariant of $P$. We include a proof for completeness:







\iftrue

\begin{proposition}
\label{res:Izmestiev_coarsens_orbit}
$G_P^{\mathfrak I}$ is at least as symmetric as $P$, that is, $\phi^{\mathfrak I}$ is well-defined.
\end{proposition}

\begin{proof}
Fix a linear symmetry $T\in\Aut_{\GL}(P)$ and let $\sigma_T\in\Aut(G_P)$ be the induced combinatorial symmetry of the edge-graph.
We need to show that $\sigma_T$ preserves the Izmestiev coloring, that is, $\sigma_T\in \Aut(G_P^\II)$.

This requires two ingredients. 
For the first, one checks that the generalized polar dual $P^\circ(c)$ (like the usual polar dual) satisfies
$$(TP)^\circ(c)=T\mT\! P^\circ(c),$$
which then gives us
\begin{equation}
\label{eq:1345}
\vol\big((TP)^\circ(c)\big)=\det(T\mT)\vol(P^\circ(c))=\vol(P^\circ(c)),
\end{equation}
where we used that $\det(T\mT)=\det(T)=1$ holds for all linear transformations in~a finite matrix group such as $\Aut_{\GL}(P)$.

The second ingredient is the following: 
\begin{align}
(TP)^\circ(c) \label{eq:123}
    &= \{ x\in\RR^d \mid \<x,Tv_i\>\le c_i\text{ for all $i\in V$}\}
  \\&= \{ x\in\RR^d \mid \<x,v_{\sigma_T(i)}\>\le c_i\text{ for all $i\in V$}\}  \nonumber
  \\&= \{ x\in\RR^d \mid \<x,v_i\>\le c_{\sigma_T^{-1}(i)}\text{ for all $i\in V$}\} \nonumber 
  \\&= P^\circ(\Pi_{\sigma_T} c).\footnotemark  \nonumber
\end{align}%
%

Putting everything together, we can show $\mathfrak I(i)=\mathfrak I(\sigma_T(i))$ for all $i\in V$, and~equivalently for edges.
We show both at the same time by proving $M_{ij}=M_{\sigma_T(i)\sigma_T(j)}$~for all $i,j\in\{1,...,n\}$:%
\footnotetext{Recall that $\Pi_\sigma$ was defined so that $(\Pi_\sigma v)_i = v_{\sigma^{-1}(i)}$ for a vector $v\in\RR^n$.}%
\begin{align*}
M_{ij}
   = \frac{\partial^2 \vol(P^\circ(c))}{\partial c_{i} \partial c_{j}} \Big|_{c=c_0}
 \!\!\!\!\!&= \frac{\partial^2 \vol(P^\circ(\Pi_\sigma c))}{\partial c_{\sigma_T(i)} \partial c_{\sigma_T(j)}} \Big|_{c=c_0}
 \!\!\!\!\!\!\!\overset{\eqref{eq:123}}= \frac{\partial^2 \vol((TP)^\circ(c))}{\partial c_{\sigma_T(i)} \partial c_{\sigma_T(j)}} \Big|_{c=c_0}
 \\[1ex]&\overset{\eqref{eq:1345}}= \frac{\partial^2 \vol(P^\circ(c))}{\partial c_{\sigma_T(i)} \partial c_{\sigma_T(j)}} \Big|_{c=c_0}
 \!\!\!\!\!
 = M_{\sigma_T(i)\sigma_T(j)},
\end{align*}
where we set $c_0 := (1,...,1)\in\RR^n$.
\end{proof}

\else

\begin{proposition}
\label{res:Izmestiev_coarsens_orbit}
The Izmestiev coloring coarsens the (linear) orbit coloring.
\end{proposition}

\begin{proof}
Fix $i,j\in V$ with the same color in the orbit coloring.
Then $v_i,v_j\in\F_0(P)$ are in the same $\Aut_{\GL}(P)$-orbit, \ie\ there is a $T\in\Aut_{\GL}(P)$ with $Tv_i=v_j$.
Let $\sigma_T\in$ $\Aut(G_P)$ be the induced symmetry of the edge-graph.

It is well known that $(TP)^\circ= T\mT\! P^\circ$, and the same holds for the generalized dual: $(TP)^\circ(c)=T\mT\! P^\circ(c)$.
In particular, $\vol((TP)^\circ(c))=\det(T\mT)\vol(P^\circ(c))=\vol(P^\circ(c))$, since $\det(T\mT)=\det(T)=1$ for all linear transformations that are contained in a finite matrix group.

Finally, one can confirm that $(TP)^\circ(c) = P^\circ(\Pi_{\sigma_T} c)$:
\begin{align*}
(TP)^\circ(c) 
    &= \{ x\in\RR^d \mid \<x,Tv_i\>\le c_i\text{ for all $i\in V$}\}
  \\&= \{ x\in\RR^d \mid \<x,v_{\sigma_T(i)}\>\le c_i\text{ for all $i\in V$}\}
  \\&= \{ x\in\RR^d \mid \<x,v_i\>\le c_{\sigma_T^{-1}(i)}\text{ for all $i\in V$}\}
  \\&= P^\circ(\Pi_{\sigma_T} c).\footnotemark
\end{align*}%
\footnotetext{Note that multiplication with a permutation matrix $\Pi_\sigma$ from the \emph{left} permutes the rows by the \emph{inverse} permutation $\sigma^{-1}$.}%
%

Now, putting everything together and abbreviating $c_0 := (1,...,1)\in\RR^n$, we find
\begin{align*}
M_{ij}(P)
   &= \frac{\partial^2 \vol(P^\circ(c))}{\partial c_{i} \partial c_{j}} \Big|_{c=c_0}
 \!\!\!\!\!= \frac{\partial^2 \vol(P^\circ(\Pi_\sigma c))}{\partial c_{\sigma_T(i)} \partial c_{\sigma_T(j)}} \Big|_{c=c_0}
 \\&= \frac{\partial^2 \vol((TP)^\circ(c))}{\partial c_{\sigma_T(i)} \partial c_{\sigma_T(j)}} \Big|_{c=c_0}
 \!\!\!\!\!= M_{\sigma_T(i)\sigma_T(j)}(TP) = M_{\sigma_T(i)\sigma_T(j)}(P).
\end{align*}
\end{proof}

\fi

\begin{theorem}
\label{res:Izmestiev_works}
The Izmestiev coloring captures the linear symmetries of $P$.
\begin{proof}
By \cref{res:Izmestiev_coarsens_orbit}, the Izmestiev coloring $\II$ is at least as symmetric as $P$, and so we can try to apply the \enquote{linear algebra criterion} (\cref{res:criterion}) to show that $\II$ captures linear symmetries.
That is, we need to show that $U:=\Span\Phi^\dagger$ is an eigenspace of the colored adjacency matrix $A^\II$ of $G_P^\II$.
Recall that $A^\II$ is exactly the Izmestiev matrix (\cref{res:Izmestiev_adjacency}), and so we can try to use the various properties of this matrix established in \cref{res:izmestiev}.


First, $U=\Span \Phi^\dagger=\Span\Phi\T$ (since the columns of $\Phi\T$ and $\Phi^\dagger$ are~dual~bases of $U$), and so \cref{res:izmestiev} \itm4 can be read as $U\subseteq\ker A^\II$.
Second, we have~both $\dim U=\rank\Phi =d$ (since $P$ is full-dimensional) and $\dim\ker A^\II=d$ (by \cref{res:izmestiev} \itm5).
Comparing dimensions, we thus have $U=\ker A^\II$.

We conclude that $U$ is an eigenspace of $A^\II$ (namely, the eigenspace to eigenvalue $0$).
The \enquote{linear algebra criterion} \cref{res:criterion} then asserts that $\II$ captures the~linear symmetries of~$P$.
\end{proof}
\end{theorem}

By \cref{res:orbit_coloring_finest}, if there is any coloring that captures linear symmetries, then the orbit coloring does so as well:

\begin{corollary}
\label{res:orbit_coloring_works}
The orbit coloring captures the linear symmetries of $P$.
\end{corollary}

\begin{remark}

A coloring $\cc$ is said to be \emph{finer} than a coloring $\bar \cc$ if 
\begin{align*}
\qquad\qquad\qquad\qquad\cc(i)=\cc(\hati)\phantom{\hatj}\;\;&\implies\;\; \phantom j\bar\cc(i)=\bar\cc(\hati),\quad &&\text{for all $i,\hati\in V$},\qquad\qquad\qquad\qquad\\
\qquad\qquad\qquad\qquad\cc(ij)=\cc(\hati\hatj)\;\;&\implies\;\; \bar\cc(ij)=\bar\cc(\hati\hatj),\quad &&\text{for all $ij,\hati\hatj\in E$}.\qquad\qquad\qquad\qquad
\end{align*}
Conversely, $\bar \cc$ is said to be \emph{coarser} than $\cc$.

It is easy to see that the orbit coloring is the \emph{finest} coloring that captures linear symmetries, that is, it uses the most colors (consider the argument in the first paragraph of \cref{ssec:orbit_coloring}).
In contrast, the Izmestiev coloring is in general neither the finest nor the coarsest coloring with this property. 
Actually determining the coarsest such coloring (\ie\ using the fewest colors) seems like a challenging task.
\end{remark}

\section{Capturing orthogonal symmetries}
\label{sec:capturing_Euclidean_symmetries}

For this section we consider the orthogonal symmetry group $\Aut_{\OO}(P)$ and all notations without an explicit hint to the kind of symmetry (such as $\phi^\cc$ or $\oo$) implicitly refer to their orthogonal versions.

Recall the metric coloring $\mathfrak m\:V\cupdot E\to\RR$ (\cref{def:metric_coloring}) with
\begin{align*}
\qquad\qquad\qquad\mathfrak m(i)&=\|v_i\|^2,\quad &&\text{for all $i\in V$},\qquad\qquad\qquad\\
\qquad\qquad\qquad\mathfrak m(ij)&=\<v_i,v_j\>,\quad &&\text{for all $ij\in E$}.\qquad\qquad\qquad
\end{align*}
As previously mentioned, we consider $\mm$ a candidate for capturing orthogonal symmetries, but we are yet unable to prove this (see \cref{q:metric_coloring}).

Nevertheless, combining the metric coloring and the Izmestiev coloring allows us to construct a coloring for which we can actually prove this.

\begin{definition}
Given two colorings $\mathfrak c\:V\cupdot E\to\mathfrak C$ and $ \bar{\mathfrak c}\:V\cupdot E\to \bar{\mathfrak C}$, the \emph{product coloring} $\cc\times\bar\cc\:V\cupdot E\to\mathfrak C\times \bar{\mathfrak C}$ is defined by
\begin{align*}
\qquad\qquad\qquad\mathfrak (\cc\times\bar\cc)(i)\phantom j&:=(\cc(i),\phantom j\bar\cc(i)),\quad &&\text{for all $i\in V$},\qquad\qquad\qquad\\
\qquad\qquad\qquad\mathfrak (\cc\times\bar\cc)(ij)&:=(\cc(ij),\bar\cc(ij)),\quad &&\text{for all $ij\in E$}.\qquad\qquad\qquad
\end{align*}
\end{definition}

The relevant (and easy to verify) property of the product coloring is 
\begin{equation}
\label{eq:product_aut}
    \Aut(G_P^{\cc\times\bar \cc})=\Aut(G_P^\cc)\cap\Aut(G_P^{\bar\cc}).
\end{equation}
In particular, if both $\phi^\cc$ and $\phi^{\bar \cc}$ are well-defined, then so is $\phi^{\cc\times\bar \cc}$.

\begin{theorem}
\label{res:capturing_Euclidean_symmetries}
The coloring $\mathfrak I\times\mathfrak m$ captures the orthogonal symmetries of $P$.
\begin{proof}
The Izmestiev coloring $\II$ is at least as symmetric as $P$ (we know this for~linear symmetries by \cref{res:Izmestiev_coarsens_orbit}, which include the orthogonal symmetries as a special case). 
Like-wise, the metric coloring $\mm$ is at least as symmetric as $P$ (every orthogonal symmetry preserves norms and inner products, and therefore also the metric coloring).
So, since $\phi^\II$ and $\phi^\mm$ are well-defined, so is $\phi^{\II\times\mm}$.


It remains to show that $\phi^{\II\times\mm}$ has an inverse.
For that, fix a $\sigma\in\Aut(G_P^{\mathfrak I\x \mathfrak m})$.~%
By \eqref{eq:product_aut} we have $\sigma\in\Aut(G_P^\II)$.
By \cref{res:Izmestiev_works} there is a corresponding \mbox{$T_\sigma\in\Aut_{\GL}(P)$} with $T_\sigma v_i =$ $ v_{\sigma(i)}$ for all $i\in V$.
It remains to show that $T_\sigma\in\OO(\RR^d)$.

Since $P$ is full-dimensional, a set $S$ that contains any vertex $v_i$ together with its neighbors $\{v_j\mid ij\in E\}$ spans $\RR^d$, and so it suffices to verify
$\<T_\sigma v_k,T_\sigma v_\ell\> = \<v_k,v_\ell\>$
for every two $v_k,v_\ell\in S$ to prove the orthogonality of $T_\sigma$.

Also by \eqref{eq:product_aut}, $\sigma$ preserves the metric coloring $\mm$.
The claim then follows via
$$
\<v_k,v_\ell\> = \<v_{\sigma(k)},v_{\sigma(\ell)}\> = \<T_\sigma v_k,T_\sigma v_\ell\>,\qquad\text{for all $v_k,v_\ell\in S$},
$$
where we used that $v_k,v_\ell\in S$ implies $k=\ell$ or $k\ell\in E$.
%
%
%
\end{proof}
\end{theorem}




By (the orthogonal version of) \cref{res:orbit_coloring_finest}, if there is any coloring that captures orthogonal symmetries, then so does the orthogonal orbit coloring:

\begin{corollary}
\label{res:Euclidean_orbit_coloring_works}
The orthogonal orbit coloring captures orthogonal symmetries.
\end{corollary}

\section{Outlook, open questions and further notes}
\label{sec:outlook}

In this article we have shown that the edge-graph of a convex polytope, while~generally a very weak representative of the polytope's geometric nature, still has sufficient structure to let us encode two important types of geometric symmetries:\nlspace linear and orthogonal symmetries.
We achieved this by coloring the vertices and edges~of the edge-graph.

The first coloring for which we established that it \enquote{captures the polytope's linear symmetries} was the \emph{Izmestiev coloring} (\cref{res:Izmestiev_works}), based on an ingenious~construction by Ivan Izmestiev.
But we also found that the \emph{orbit coloring}, a conceptu\-ally very easy coloring, does the job as well (\cref{res:orbit_coloring_works}).
Analogous colorings~exist for the orthogonal symmetries as well (\cref{res:capturing_Euclidean_symmetries} and \cref{res:Euclidean_orbit_coloring_works}).

In the following we briefly discuss various potential generalizations and follow up questions concerning these results.
This further highlights the very special~structure of convex polytopes that went into our theorems, emphasizing again that these~results are non-trivial to achieve and to generalize.

We also want to mention the following neat consequence for \enquote{very symmetric} polytopes:


\begin{corollary}
If $P\subset\RR^d$ is vertex- and edge-transitive (\ie\ its linear \resp~ortho\-gonal symmetry group has a single orbit on vertices and edges), then $P$ is exactly~as symmetric as its edge-graph. 
\end{corollary}

This observation has previously been made in \cite[Theorem 5.2]{winter2020polytopes}.
No classification of simultaneously vertex- and edge-transitive polytopes is known so far, and so this fact might help in the study of this class.

\subsection{Capturing other types of symmetries}

Besides linear and orthogonal symmetries, there are at least two further common groups of symmetries associated with a polytope: the \emph{projective symmetries} and the \emph{combinatorial symmetries} (that is, the symmetries of the face lattice). 

We can ask whether those too can be captured by a colored edge-graph:

\begin{question}
Is there a coloring $\cc:V\cupdot E\to\mathfrak C$ that captures projective \resp\ combinatorial symmetries: 
$$\Aut(G_P^\cc)\cong\Aut_{\mathrm{PGL}}(P)\quad\text{\resp}\quad \Aut(G_P^\cc)\cong\Aut_{\mathrm{Comb}}(P)\;?$$
%
\end{question}

There might be a general strategy derived from the following (informal) inclusion chain of the symmetry groups:
$$
\Aut_{\OO}(P)
\;\subseteq\; \Aut_{\GL}(P)
\;\subseteq\;\Aut_{\mathrm{PGL}}(P)
\;\;\text{\enquote{$\subseteq$}}\;\Aut_{\mathrm{Comb}(P)}.
$$
As it turns out, having solved the coloring problem further to the left in the chain can help to solve the problem further to the right -- at least to some degree.

For example, note that every polytope $P$ can be linearly transformed via~a~trans\-formation $T\in\GL(\RR^d)$ so that $\Aut_{\GL}(P)=\Aut_{\OO}(TP)$.
That is, a coloring of $G_P$ that captures the orthogonal symmetries of $TP$ (which has the same edge-graph) also captures the linear symmetries of $P$.
In still other words, we solved the problem of capturing linear symmetries by making use of our ability to capture orthogonal symmetries.

In our approach, we have not made use of this because we needed to solved the linear case before the orthogonal one.
However, this can be of use for capturing projective symmetries.
More explicitly, the question is as follows:
for every polytope $P$, is there a projective transformation $T\in\PGL(\RR^d)$ so that $\Aut_{\PGL}(P)=\Aut_{\GL}(TP)$?

The same approach seems doomed for capturing combinatorial symmetries: there are polytopes with combinatorial symmetries that cannot be realized geometrically (\!\cite{bokowski1984combinatorial} discusses the case of a combinatorial symmetry that cannot be made linear; to our knowledge, realizing them as projective symmetries remains to be discussed).

\subsection{Edge-only coloring}

For capturing the symmetries of certain 2-dimensional polytopes it is necessary to color both vertices \emph{and} edges (\cf\ \cref{fig:hex_symmetry_2}).
But~it~is~unclear whether this is still necessary in higher dimensions.

\begin{question}
Is it sufficient to color \emph{only the edges} if $d\ge 3$? That is, is there an edge-only coloring $\cc\:E\to\mathfrak C$ that captures (for example) linear symmetries?
\end{question}

A vertex-only coloring is not always sufficient. For example, in even dimensions exist vertex-transitive neighborly polytopes other than the simplex: \eg\ for $n\ge 6$ we have the following cyclic 4-polytope with $n$ vertices that is not a simplex:
%
$$P:=\conv\left\{\begin{pmatrix}\cos(2\pi i/n)\\\sin(2\pi i/n)\\\cos(\phantom4\pi i/n)\\\sin(\phantom4\pi i/n)\end{pmatrix}\in\RR^4\;\Bigg\vert\; i\in\{1,...,n\}\right\}.$$
The edge-graph of $P$ is the complete graph $K_n$, and $P$ has a single orbit of vertices.
Thus, if $\cc\: V\to\mathfrak C$ is a vertex-only coloring that captures the symmetries of $P$,~then all vertices of $K_n$ must receive the same color.
But if the edges receive no color, then $\Aut(K_n^\cc)$ $=\Sym(V)$.
However, it is known that the linear symmetry group of the cyclic polytope $P$ other than a simplex is strictly smaller than $\Sym(V)$ \cite{KaibelAutomorphismGO}.

\subsection{Non-convex polytopes and general graph embeddings}

Our approach suggests no immediate generalization to non-convex polytopes or various forms of polytopal complexes.

\begin{question}
What is the most general geometric setting in which the symmetries can be \enquote{captured} by coloring the edge-graph? Does it work for non-convex and/or self-intersecting polytopes? What about more general polytopal complexes?
\end{question}

A vast generalization of polytope skeleta are \emph{graph embeddings}.
For a graph~$G=$ $(V,E)$, a graph embedding is simply a map $v\:V\to\RR^d$.
There are natural notions of symmetry for such embeddings, and so one might ask whether it is possible to \enquote{capture} them by coloring the graph.
The following example shows that this is not possible in general:

\begin{example}
\label{ex:K_4_4}
Consider the complete bipartite graph $K_{4,4}$ with vertex set $V_1\cupdot V_2=\{1,2,3,4\}\cupdot\{5,6,7,8\}$ and an embedding into $\RR^4$ defined as follows:
\begin{align*}
v_1 = (+1,0,0,0),\qquad v_5 = (0,0,+1,0), \\
v_2 = (0,+1,0,0),\qquad v_6 = (0,0,0,+1),\\
v_3 = (-1,0,0,0),\qquad v_7 = (0,0,-1,0),\\
v_4 = (0,-1,0,0),\qquad v_8 = (0,0,0,-1).
\end{align*}
One can check that the linear symmetry group of this embedding acts transitively on the vertices as well as the edges.
Thus, a coloring $\cc$ that is at least as symmetric as the graph embedding must assign the same color to all vertices, and like-wise, the same color to all edges.
That is, $\Aut(K_{4,4}^\cc)=\Aut(K_{4,4})$.

However, one can also see that the given embedding has a strictly smaller symmetry group than $\Aut(K_{4,4})$.
For example, $\sigma:=(12)\in\Aut(K_{4,4})$ cannot be realized as a geometric symmetry.
\end{example}

It might be interesting to determine conditions under which \enquote{capturing symmetries} is possible even in this very general case.









\subsection{The metric coloring}
\label{ssec:metric_coloring_open}

It is yet unknown whether the metric coloring alone can capture orthogonal symmetries (\cf\ \cref{ssec:metric_coloring} and \cref{sec:capturing_Euclidean_symmetries}).

\begin{question}
\label{q:metric_coloring}
Can the metric coloring $\mathfrak m$ capture orthogonal symmetries?
\end{question}

Any potential affirmative answer to \cref{q:metric_coloring} will need to make use of similar assumptions as the construction of the Izmestiev coloring, namely, convexity and $0\in\mathrm{int}(P)$, as there are known counterexamples for the other cases (see \cref{fig:metric_counterexample} and \cref{fig:metric_counterexample_2}).




\begin{figure}[h!]
    \centering
    \includegraphics[width=0.55\textwidth]{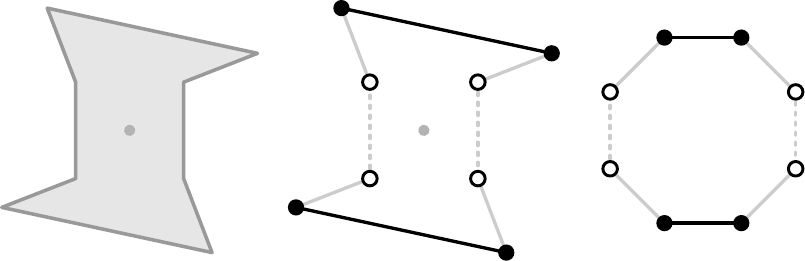}
    \caption{A non-convex shape and two drawings of its edge-graph with metric coloring. 
    The colored edge-graph has more symmetries than the polygon.}
    \label{fig:metric_counterexample}
\end{figure}

\begin{figure}[h!]
    \centering
    \includegraphics[width=0.55\textwidth]{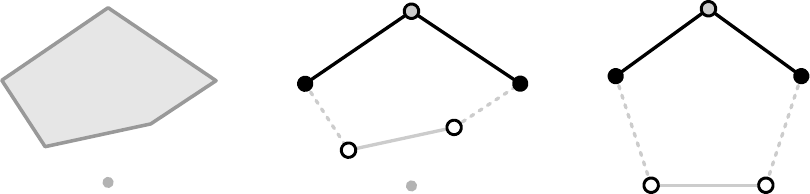}
    \caption{A convex polygon $P$ with $0\not\in\Int(P)$  (the gray dot indicates the origin) and two drawings of its edge-graph with metric coloring. 
    The colored edge-graph has more symmetries than the polygon.}
    \label{fig:metric_counterexample_2}
\end{figure}

An interesting special case is the following:

\begin{question}
If $P$ is inscribed (\ie\ it has all its vertices on a common sphere around the origin) and has all edges of the same length, then is it true that $P$ is as symmetric as its edge-graph, that is, $\Aut_{\OO}(P)\cong\Aut(G_P)$?
\end{question}

\par\bigskip
\parindent 0pt
\textbf{Acknowledgments.} The author thanks Frank Göring (TU Chemnitz) for insightful discussions.


\bibliographystyle{abbrv}
\bibliography{literature}

\end{document}